\documentclass[11pt]{article}

\usepackage{amsmath, amsthm, amssymb}

\usepackage{tikz}
\usetikzlibrary{decorations,decorations.pathmorphing}

\date{}
\title{\vspace{-0.8cm}Pancyclic subgraphs of random graphs}
\author{
Choongbum Lee \thanks{Department of Mathematics, UCLA, Los
Angeles, CA, 90095. Email: choongbum.lee@gmail.com. Research
supported in part by Samsung Scholarship.}
\and
Wojciech Samotij \thanks{Department of Mathematics,
University of Illinois, Urbana, IL, 61801, USA. E-mail address:
samotij2@illinois.edu. Research supported in part by Schark Fellowship and Parker Fellowship.} }

\newcommand{\abs}[1]{\lvert #1 \rvert}

\newcommand{\Ex}{\mathbb{E}}
\newcommand{\Pd}{{\cal P}_\delta}
\newcommand{\modulon}{[\underline{n}]}

\DeclareMathOperator{\Bi}{Bi}

\theoremstyle{plain}
\newtheorem{THM}{Theorem}[section]
\newtheorem{PROP}[THM]{Proposition}
\newtheorem{LEMMA}[THM]{Lemma}
\newtheorem{COR}[THM]{Corollary}
\newtheorem{CLAIM}[THM]{Claim}

\newtheorem*{REM}{Remark}

\theoremstyle{definition}
\newtheorem{DFN}[THM]{Definition}

\numberwithin{figure}{section}

\begin{document}
\maketitle

\begin{abstract}
An $n$-vertex graph is called pancyclic if it contains a cycle of length $t$
for all $3 \leq t \leq n$. In this paper, we study pancyclicity of random graphs
in the context of resilience, and prove that if $p \gg n^{-1/2}$, then
the random graph $G(n,p)$ a.a.s.~satisfies the following property: Every Hamiltonian subgraph
of $G(n,p)$ with more than $(\frac{1}{2} + o(1)){n \choose 2}p$ edges
is pancyclic. This result is best possible in two ways. First, the range of $p$ is asymptotically tight;
second, the proportion $\frac{1}{2}$ of edges cannot be reduced. Our theorem
extends a classical theorem of Bondy, and is closely related to a
recent work of Krivelevich, Lee, and Sudakov. The proof 
uses a recent result of Schacht (also independently obtained by Conlon and Gowers).
\end{abstract}

\section{Introduction}
\label{section_introduction}

The problem of finding cycles of various lengths in graphs is one
of the main topics of study in extremal graph theory.
For a fixed odd integer $l$ with $l \geq 3$, Erd\H{o}s and Stone \cite{ErSt} proved
that an $n$-vertex graph with more than $(\frac{1}{2} + o(1)){n \choose 2}$ edges
contains a cycle of length $l$, and for an even integer $l$, Bondy and
Simonovits \cite{BoSi} proved that $c n^{1+2/l}$ edges suffice,
for some constant $c$ depending on $l$. It is remarkable that the
same graph parameter can behave so differently according to the parity of $l$.
When the length of the cycle grows with the order of the graph,
the most studied problem is that of finding a Hamilton cycle, i.e.,
a cycle passing through every vertex of the graph. This can be a very
difficult problem since determining whether a given graph is Hamiltonian is an NP-complete problem.
However, there are some simple sufficient conditions which force a graph
to be Hamiltonian. For example, Dirac (see, e.g., \cite{Diestel}) proved
that a graph on $n$ vertices with minimum degree at least
$\left\lceil \frac{n}{2} \right\rceil$ contains a Hamilton cycle.

An $n$-vertex graph is called {\em pancyclic} if it contains a cycle of length $t$ for all
$3 \leq t \leq n$. Clearly, every pancyclic graph is also Hamiltonian and the converse is not true.
Quite surprisingly though, many times a condition which forces a
graph to be Hamiltonian also forces it to be pancyclic.
For instance, Bondy~\cite{Bondy2} extended Dirac's theorem and
proved that a graph on $n$ vertices with minimum degree greater than
$\left\lceil \frac{n}{2} \right\rceil$ is pancyclic. In fact, he proved the
following stronger statement.

\begin{THM}
  \label{thm_Bondy}
  Every Hamiltonian graph with $n$ vertices and more than $\frac{1}{2}{n \choose 2}$ edges is pancyclic.
\end{THM}

For more results on cycles in graphs, see~\cite{Bondy1}.

The results mentioned above can be also stated in the framework of {\em resilience},
defined by Sudakov and Vu \cite{SuVu}.

\begin{DFN} Let $\mathcal{P}$ be a monotone increasing graph property and $G$ be a graph having $\mathcal{P}$
\begin{itemize}
  \setlength{\itemsep}{1pt}
  \setlength{\parskip}{0pt}
  \setlength{\parsep}{0pt}
\item[(i)] (Global resilience) The {\em global resilience} of $G$ with respect to $\mathcal{P}$ is the minimum number $r$ such that by deleting $r$ edges from $G$ one can obtain a graph not having $\mathcal{P}$.
\item[(ii)] (Local resilience) The {\em local resilience} of $G$ with respect to $\mathcal{P}$ is the minimum number $r$ such that by deleting at most $r$ edges at each vertex of $G$ one can obtain a graph not having $\mathcal{P}$.
\end{itemize}
\end{DFN}

In this framework, Dirac's theorem can be reformulated as:
``The complete graph on $n$ vertices $K_n$ has local resilience $\left\lfloor \frac{n}{2} \right\rfloor$ with respect to Hamiltonicity''.
Theorem~\ref{thm_Bondy} can also be formulated as a global resilience result with an additional constraint:
``If one deletes fewer than $\frac{1}{2}{n \choose 2}$ edges from $K_n$ while preserving Hamiltonicity, then the resulting graph is always pancyclic''.
Sudakov and Vu pointed out that many classical results in extremal graph theory can be viewed
in the context of resilience, and initiated a systematic
study of resilience of random graphs. For example, in~\cite{SuVu} they showed that if
$p > \frac{\log^4 n}{n}$, then a.a.s.~every subgraph of $G(n,p)$ with minimum
degree at least $(\frac{1}{2}+o(1))np$ is Hamiltonian. Note that this is an extension
of Dirac's theorem, since $K_n$ can be regarded as the random graph $G(n,1)$.
For more results of this flavor,
see \cite{BaCsSa,BeKrSu,BoKoTa,DeKoMaSt,FrKr,HaLeSu,KrLeSu,SuVu}.

In this paper, we study the global resilience of random graphs with respect
to pancyclicity by extending Theorem~\ref{thm_Bondy} to binomial random graphs.
The {\em binomial random graph} $G(n,p)$ denotes the probability distribution on the
set of all graphs with vertex set $\{1,\ldots,n\}$ such that each pair of
vertices forms an edge randomly and independently with probability
$p$. We say that $G(n,p)$ possesses a graph property $\mathcal{P}$
{\em asymptotically almost surely} (a.a.s.~for short) if the
probability that a graph drawn from $G(n,p)$ possesses $\mathcal{P}$ tends to $1$ as $n$
tends to infinity. The pancyclicity
of random graphs has been studied in several papers, including
 \cite{Cooper1, Cooper2, CoFr, KrLeSu, Luczak}.
Krivelevich, Lee, and Sudakov~\cite{KrLeSu} proved that if $p \gg n^{-1/2}$, then $G(n,p)$ a.a.s.~has local resilience
$(\frac{1}{2} + o(1))np$ with respect to being pancyclic.
Since pancyclicity is a global property, it is more meaningful
to measure the local resilience than the global resilience; in a `typical' graph from $G(n,p)$,
one can easily remove all
Hamilton cycles by deleting about $np$ edges incident to a single vertex and
isolating that vertex. However, similarly as in Theorem~\ref{thm_Bondy},
if one adds the additional constraint that the remaining
graph should be Hamiltonian, then removing all cycles of some length $t \in \{3, \ldots, n-1\}$
requires deleting many more edges. We establish the global resilience of random
graphs with respect to pancyclicity in this sense.
Our main result is the following extension of Theorem~\ref{thm_Bondy}.

\begin{THM} \label{thm_mainthm}
If $p \gg n^{-1/2}$, then $G(n,p)$ a.a.s.~satisfies the following. Every Hamiltonian subgraph
$G' \subset G(n,p)$ with more than $(\frac{1}{2}+o(1))\frac{n^2p}{2}$ edges is pancyclic.
\end{THM}
\begin{REM}
Theorem \ref{thm_mainthm} is in fact stronger than the local resilience result
obtained by Krivelevich, Lee, and Sudakov due to Sudakov and Vu's result
on resilient Hamiltonicity of random graphs. Let $p \gg n^{-1/2}$ and $G'$ be
a spanning subgraph of $G(n,p)$ with minimum degree at least $(\frac{1}{2} + o(1))np$. Then
$e(G') \geq (\frac{1}{2}+o(1))\frac{n^2p}{2}$ and, by Sudakov and Vu's theorem, $G'$ is Hamiltonian.
Therefore, $G'$ must be pancyclic by Theorem \ref{thm_mainthm} and this establishes the local resilience.
\end{REM}

Theorem \ref{thm_mainthm} is best possible in two ways. First, the range of $p$ is asymptotically tight.
To see this, assume that $p \ll n^{-1/2}$ and fix a Hamilton cycle $H$ in $G(n,p)$.
Next, from every triangle in $G(n,p)$, remove one edge which does not belong to $H$
 (this can be done as long as $n > 3$). Since a.a.s.~there are at most
$n^3p^3=o(n^2p)$ triangles in the graph, only a small proportion of edges is
removed and the resulting graph is triangle-free, thus not pancyclic.
Second, Hamilton subgraphs with fewer than $(\frac{1}{2}+o(1))\frac{n^2p}{2}$ edges need not be pancyclic. Assume that
$p \gg n^{-1/2}$ and again fix a Hamilton cycle $H$ in $G(n,p)$. If $n$ is even, then
properly color the vertices of $H$ with two colors and let $V_1, V_2$ be the color classes.
Remove all the edges of $G(n,p)$ within $V_1$ and $V_2$. A.a.s.~we will remove at most
$(\frac{1}{2} + o(1))\frac{n^2p}{2}$ edges, and the resulting graph will be bipartite, hence
not pancyclic. If $n$ is odd, then color the vertices of $H$ with two colors so that there is only one
(monochromatic) edge $e$ in the color class $V_1$ and no edges in the other color class $V_2$.
Assume that $e = xy$ and let $e_x$ and $e_y$ be the edges of $H - e$ incident to $x$ and $y$, respectively. 
Remove all the edges from $G(n,p)$ within $V_1$ except $e$, all the edges within $V_2$, and all the edges incident
to $e$ except $e_x$ and $e_y$ to obtain a graph $G'$. Note that all the odd cycles in $G'$ contain the edge $e$;
and since we removed most of the edges incident to $e$, there will be no triangles in $G'$.
Moreover, $H \subset G'$ so $G'$ is Hamiltonian, and a.a.s.~the above process produces a graph $G'$ with
$e(G') \ge (\frac{1}{2} + o(1))\frac{n^2p}{2}$. This shows that
the proportion $\frac{1}{2}$ in Theorem \ref{thm_mainthm} is optimal.

To simplify the presentation, we often omit floor 
and ceiling signs whenever these are not crucial
and make no attempts to optimize absolute constants 
involved. We also assume that the order $n$ of all
graphs tends to infinity and therefore is sufficiently 
large whenever necessary. \\

\noindent\textbf{Notation.} Let $G$ be a graph with vertex
set $V$ and edge set $E$. For a vertex $v$, its degree will be denoted by $\deg(v)$,
and for a set $X$, its total degree $\deg(X)$ is defined as
$\deg(X) = \sum_{v \in X} \deg(v)$. The maximum degree and the minimum degree of $G$ are denoted
by $\Delta(G)$ and $\delta(G)$, respectively. For a set $X \subset V$,
let $N(X)$ be the collection of all vertices $v$ which are adjacent
to at least one vertex in $X$. If $X=\{v\}$, we
denote its neighborhood by $N(v)$. Let $N^{(0)}(v) = \{ v \}$ and
$N^{(k)}(v)$ be the set of vertices at distance exactly $k$ from $v$. This
can also be recursively defined as $N^{(k)}(v) = N(N^{(k-1)}(v))
\backslash (N^{(k-1)}(v) \cup N^{(k-2)}(v))$. Note that $N^{(1)}(v)
= N(v)$. For a set $X$, we denote by $E(X)$ the set of edges in the
induced subgraph $G[X]$, and let $e(X)= |E(X)|$. Similarly,
for two sets $X$ and $Y$, we denote by $E(X,Y)$ the set of ordered
pairs $(x,y) \in E$ such that $x \in X$ and $y \in Y$, and let $e(X,Y)
= |E(X,Y)|$. Note that $e(X,X) = 2e(X)$. Sometimes we will use
subscripts to prevent ambiguity. For example, $N^{(k)}_G(v)$ is
the $k$-th neighborhood of $v$ in the graph $G$. 
The set $\modulon$ is
the set of remainders modulo $n$, and we use an integer $i$ to represent
the modulo class $i \mod n$ (one modulo class can have several different representations).

\section{Preliminaries}
\label{section_preliminaries}

In this section, we collect several known results which are later used in the proof of the main theorem.
First theorem, proved by Haxell, Kohayakawa, and {\L}uczak~\cite{HaKoLu}, establishes
the global resilience of $G(n,p)$ with respect to the property of
containing a cycle of fixed odd length.

\begin{THM} \label{thm_resilientsmallcycle}
For any fixed odd integer $l$ with $l \geq 3$, and positive $\varepsilon$, there exists a
constant $C=C(l, \varepsilon)$ such that if $p \geq C n^{-1+1/(l-1)}$, then $G(n,p)$ a.a.s.~satisfies
the following property: Every subgraph $G' \subset G(n,p)$ with at least $(\frac{1}{2} + \varepsilon)\frac{n^2p}{2}$
edges contains a cycle of length $l$.
\end{THM}

The following extremal result was proved by Bondy and Simonovits \cite{BoSi}.
\begin{THM} \label{thm_bondysimonovits}
Let $k$ be a positive integer and $G$ be a graph on $n$ vertices with more than $90kn^{1 + 1/k}$ edges. Then $G$ contains a cycle of length $2k$.
\end{THM}


One of the key ingredients in the proof of Theorem~\ref{thm_mainthm} is a recent very general result of Schacht~\cite{Schacht} (independently obtained by Conlon and Gowers~\cite{CoGo}) that allows one to transfer various extremal results about discrete structures (such as graphs) to the setting of random binomial structures (such as binomial random graphs). Before we state this theorem, we need a few definitions on hypergraphs.

\begin{DFN}
\label{dfn_alphadense}
  Let $H$ be a $k$-uniform hypergraph, let $\alpha$ and $\varepsilon_0$ be reals, and let $f \colon (0,1) \rightarrow (0,1)$ be a non-decreasing function. We say that $H$ is {\em $(\alpha, f, \varepsilon_0)$-dense} if for every $\varepsilon \ge \varepsilon_0$ and $U \subset V(H)$ with $|U| \geq (\alpha + \varepsilon)|V(H)|$, we have
\[
|E(H[U])| \geq f(\varepsilon) |E(H)|.
\]
\end{DFN}

For a $k$-uniform hypergraph $H$, $i \in \{1, \ldots, k-1\}$, $v \in V(H)$, and $U \subset V(H)$, we denote by $\deg_i(v,U)$ the number of edges of $H$ containing $v$ and at least $i$ vertices in $U \setminus \{v\}$. More precisely,
\[
\deg_i(v,U) = |\{ A \in E(H) \colon |A \cap (U \setminus \{v\})|  \geq i \text{ and } v \in A \}|.
\]
Given an arbitrary set $A$ and a real number $q \in [0,1]$, we denote by $A_q$ the random subset of $A$, where each element is included with probability $q$ independently of all other elements.

\begin{DFN}
\label{dfn_Kpbounded} 
 Let $H$ be a $k$-uniform hypergraph, let $p \in (0,1)$, and let $K$ be a positive real. We say that $H$ is $(K,p)$-bounded if for every $i \in \{1, \ldots, k-1\}$ and $q \in [p,1]$, $H$ satisfies
\[
\Ex\left[ \sum_{v \in V(H)} \deg_i(v,V(H)_q)^2 \right] \leq Kq^{2i} \frac{|E(H)|^2}{|V(H)|}.
\]
\end{DFN}

The following statement can be easily read out from the proof of~\cite[Lemma 3.4]{Schacht}.

\begin{THM}
 \label{thm_transference} 
Let $(p_n)_{n \in \mathbb{N}}$ be a sequence of probabilities and let $(v_n)_{n \in \mathbb{N}}$ and $(e_n)_{n \in \mathbb{N}}$ be sequences of integers such that $p_n v_n \rightarrow \infty$ and $p_n^k e_n \rightarrow \infty$ as $n \rightarrow \infty$. Fix $\alpha$ and $K$ with $\alpha \geq 0$ and $K \ge 1$, and let $f \colon (0,1) \rightarrow (0,1)$ be a fixed non-decreasing function. For every positive $\varepsilon$, there exist $b, \varepsilon_0, C$, and $n_0$ such that for every $n$ and $q$ with $n \ge n_0$ and $n^{-1/3} \ge q \ge Cp_n$, the following holds. 

If $H$ is an $(\alpha, f, \varepsilon_0)$-dense and $(K,p_n)$-bounded hypergraph with $|V(H)| \geq v_n$ and $|E(H)| \ge e_n$, then with probability at least $1 - e^{-bqv_n}$, every subset $W$ of $V(H)_{q}$ with $|W| \ge (\alpha + \varepsilon)|V(H)_{q}|$ contains an edge of $H$.
\end{THM}
\begin{REM}
The upper bound $n^{-1/3}$ on $q$ is a function that we chose for convenience.
In fact, it can be replaced by $1/\omega(n)$ for an arbitrary function $\omega(n)$ which tends to 
infinity as $n$ goes to infinity. Note that $n_0$ will depend on this choice of $\omega(n)$.
\end{REM}

Next theorem is a well-known concentration result (see, e.g.,
\cite[Theorem 2.3]{JaLuRu}) that will be used several times
in the proof. We denote by $\Bi(n,p)$ the binomial random variable with parameters
$n$ and $p$.
\begin{THM}[Chernoff's inequality]
  \label{thm_Chernoff}
  If $X \sim \Bi(n,p)$ and $\varepsilon>0$, then
  \[ P\big(|X - \mathbb{E}[X]| \geq \varepsilon \mathbb{E}[X]\big) \leq e^{-\Omega_\varepsilon(\mathbb{E}[X])}. \]
\end{THM}

The argument of our proof becomes simpler if we assume that the edge probability
$p$ is exactly $C n^{-1/2}$ for some positive constant $C$.
However, if $p' \leq p$, then the fact that $G(n,p')$ has
global resilience $(\frac{1}{2} - \varepsilon)\frac{n^2p'}{2}$ with respect to some property
does not necessarily imply that $G(n,p)$ has global resilience $(\frac{1}{2} - \varepsilon)\frac{n^2p}{2}$
with respect to the same property. The following proposition, which is a global resilience
form of a similar statement, \cite[Proposition 3.1]{KrLeSu}, will allow us
to overcome this difficulty by establishing the monotonicity of (relative) global
resilience. The proof is essentially the same as the proof of~\cite[Proposition 3.1]{KrLeSu};
we repeat it here for the sake of completeness.

\begin{PROP} \label{prop_probrestriction}
Assume that $0 < p' \leq p \leq 1$ and $n^2p' \gg 1$. If $G(n, p')$ a.a.s.~has global resilience at
least $(\frac{1}{2} - \frac{\varepsilon}{2})\frac{n^2p'}{2}$ with respect to a monotone increasing
graph property $\mathcal{P}$, then $G(n,p)$ a.a.s.~has global
resilience at least $(\frac{1}{2} - \varepsilon)\frac{n^2p}{2}$ with respect to the same
property.
\end{PROP}
\begin{proof}
Let $\mathcal{R(P)}$ be the property of having global resilience at
least $(\frac{1}{2} - \frac{\varepsilon}{2})\frac{n^2p'}{2}$ with respect to
property $\mathcal{P}$. Define $q = p' / p$
and consider the following two-round process of exposing the edges
of $G(n,p')$. In the first round, every edge appears with
probability $p$ (call this graph $G_1$). In the second round,
every edge that appeared in the first round will remain with
probability $q$ and will be deleted with probability $1-q$ (call
this graph $G_2$). Then, $G_1$ has the same distribution as $G(n,p)$
and $G_2$ has the same distribution as $G(n,p')$. By our assumption,
we know that $G_2$ a.a.s.~has property $\mathcal{R(P)}$. Now, let $X$
be the event that $G_1$ satisfies $P(G_2 \notin \mathcal{R(P)} |
G_1) \geq \frac{1}{2}$. Then, $\frac{1}{2}P(X) \leq P(G_2 \notin \mathcal{R(P)})=o(1)$
and therefore $P(X) = o(1)$. Thus, a.a.s.~in $G(n,p)$, $P(G_2 \notin
\mathcal{R(P)} | G_1) < \frac{1}{2}$ or in other words, $P(G_2 \in \mathcal{R(P)} |
G_1) \geq \frac{1}{2}$. Let $\mathcal{A}$ be the collection of graphs
$G_1$ having this property.

Since a.a.s.~$G(n,p) \in \mathcal{A}$, we can condition on the event that $G_1 = G(n,p) \in \mathcal{A}$.
Given a subgraph $H \subset G_1$ with at
most $(\frac{1}{2} - \varepsilon)\frac{n^2p}{2}$ edges, sample every edge of $G_1$ with probability $q$ to obtain
subgraphs $H' \subset G_2 \subset G_1$. Since $G_1 \in \mathcal{A}$, we know that
$P(G_2 \in \mathcal{R(P)} | G_1) \geq \frac{1}{2}$; also, by Theorem~\ref{thm_Chernoff},
$P\big( e(H') \leq (\frac{1}{2} - \frac{\varepsilon}{2})\frac{n^2p'}{2}\big) \geq 1-o(1)$.
Thus, these two events have a non-empty intersection and therefore it is possible to find
subgraphs $H' \subset G_2 \subset G_1$ such that $G_2 \in \mathcal{R(P)}$ and
$e(H') \leq (\frac{1}{2} - \frac{\varepsilon}{2})\frac{n^2p'}{2}$. Then $G_2 - H'$ must have
property $\mathcal{P}$, and hence $G_1 - H$ must also have it by monotonicity.
\end{proof}

\section{Main Theorem}
\label{section_mainthm}

In this section, we prove the following main theorem.
\begin{THM} \label{thm_mainthm2}
For all positive $\varepsilon $, there exists a constant $C$ such that if $p \geq Cn^{-1/2}$, then $G(n,p)$
a.a.s.~satisfies the following: Every Hamiltonian subgraph
$G'$ of $G(n,p)$ with more than $(\frac{1}{2}+\varepsilon)\frac{n^2p}{2}$ edges is pancyclic.
\end{THM}

The proof breaks down into two parts. Let $\delta$ be a positive constant depending on $\varepsilon$
which will be chosen in the first part of the proof (one can take $\delta = \varepsilon / 64$).
First, we show that $G'$ contains very short cycles (lengths from $3$ to $\delta n$) 
and very long cycles (lengths from $(1-\delta)n$ to $n$). 
Second, we show that $G'$ contains all cycles of 
intermediate lengths (from $\delta n$ to $(1-\delta)n$). 
Since the ideas and tools used in each part are quite different,
we split the proof into two subsections. 


\subsection{Short and Long Cycles}

\begin{THM} \label{thm_mainthmpart3}
For all positive $\varepsilon$, there exist positive constants $\delta$ and $C$
such that if $p \geq Cn^{-1/2}$, then $G(n,p)$
a.a.s.~satisfies the following: Every Hamiltonian subgraph
$G'$ of $G(n,p)$ with more than $(\frac{1}{2}+\varepsilon)\frac{n^2p}{2}$ edges contains
a cycle of length $t$ for all $t \in [3, \delta n] \cup [(1-\delta)n, n]$.
\end{THM}

If we choose $C$ to be large enough, then the existence of cycles of lengths 3 to 7
directly follows from 
Haxell, Kohayakawa, and {\L}uczak's Theorem \ref{thm_resilientsmallcycle} (odd cycles),
and Bondy and Simonovits' Theorem \ref{thm_bondysimonovits} (even cycles). Thus, in the remainder
of this section we will only focus on cycles of length 8 and above.

Fix a labeling of the vertices of the complete graph $K_n$ with numbers from the set $\modulon$ of remainders modulo $n$. Throughout this section, addition of the elements of $\modulon$ will be performed modulo $n$. Let $C_n$ be the subgraph of $K_n$ consisting of the edges $\{i, i+1\}$ for all $i \in \modulon$. In order to better visualize our argument, we will from now on assume that $C_n$ is drawn as a circle and the vertices $0, \ldots, n-1$ lie in the clockwise order on that circle. Moreover, for each $k \in \modulon$, we denote its distance from $0$ on that circle by $\|k\|$; more precisely, let
\[
\| k \| = \min \left\{ x \geq 0 \colon x \equiv k \mod n \text{ or } x \equiv -k \mod n \right\}.
\]
Given an integer $l$ with $0 \leq l \leq \frac{n}{2}$, we will call a subgraph $X$ of $K_n \setminus C_n$ an $l$-shortcut if it is of one of the following types:
\begin{itemize}
  \setlength{\itemsep}{1pt}
  \setlength{\parskip}{0pt}
  \setlength{\parsep}{0pt}
\item[(i)]
  There are $i_1, i_2, i_3, i_4 \in \modulon$, such that $i_1, i_1 + 1, i_2, i_2 + 1, i_3, i_3 + 1, i_4$ and $i_4 + l + 1$ are all distinct and lie in the clockwise order on the cycle $C_n$, and $X$ consists of the edges $\{i_1, i_3\}, \{i_1 + 1, i_4\}, \{i_2, i_4 + l + 1\}$, and $\{i_2 + 1, i_3 + 1\}$.
\item[(ii)]
  There are $i_1, i_2, i_3, i_4 \in \modulon$, such that $i_1, i_1 + 1, i_2, i_2 + 1, i_4, i_4 + l + 1, i_3$ and $i_3 + 1$ are all distinct and lie in the clockwise order on the cycle $C_n$, and $X$ consists of the edges $\{i_1, i_3\}, \{i_1 + 1, i_4\}, \{i_2, i_4 + l + 1\}$, and $\{i_2 + 1, i_3 + 1\}$.
\end{itemize}
Observe that for every $l \in \{0, \ldots, \frac{n}{2}\}$ and an $l$-shortcut $X$, the graph $C_n \cup X$ contains cycles of lengths $l + 8$ and $n - l$, see Figure~\ref{fig:shortcuts}.

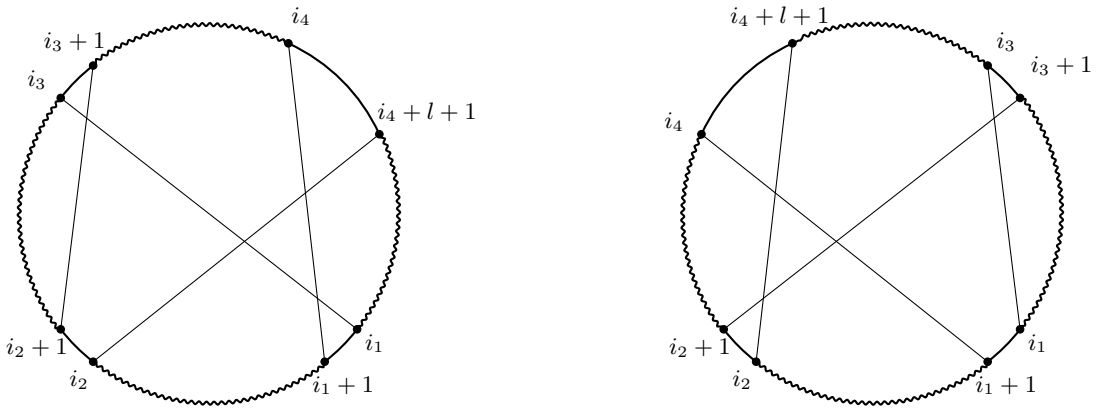
\begin{figure}[b]
  \centering
  \begin{tabular}{ccc}

\begin{tikzpicture}
  \foreach \a / \r / \j in {
    128/29/i_3+1, 142/29/i_3, 
    218/29/i_2+1, 232/28/i_2,
    308/29/i_1+1, 322/28/i_1,
    25/32/i_4+l+1, 65/29/i_4
  }
  {
    \draw [fill=black] (\a:25mm) circle (0.5mm);
    \draw (\a:\r mm) node {\footnotesize{$\j$}};
  }

  \foreach \a / \b in {142/218, 232/308, -38/25, 65/128}
  {
    \draw [thick, decorate, decoration={snake, amplitude=.2mm, segment length=1mm}] (\a:25mm) arc (\a:\b:25mm);
  }

  \foreach \a / \b in {128/142, 218/232, 308/322, 25/65}
  {
    \draw [thick] (\a:25mm) arc (\a:\b:25mm);
  }

  \foreach \a / \b in {-38/142, 308/65, 218/128, 232/25}
  {
    \draw (\a:25mm) -- (\b:25mm);
  }  
\end{tikzpicture} & \hspace{0.5in} & 

\begin{tikzpicture}
  \foreach \a / \r / \j in {
    115/29/i_4+l+1, 155/29/i_4, 
    218/29/i_2+1, 232/28/i_2,
    308/29/i_1+1, 322/28/i_1,
    38/32/i_3+1, 52/29/i_3
  }
  {
    \draw [fill=black] (\a:25mm) circle (0.5mm);
    \draw (\a:\r mm) node {\footnotesize{$\j$}};
  }

  \foreach \a / \b in {155/218, 232/308, -38/38, 52/114}
  {
    \draw [thick, decorate, decoration={snake, amplitude=.2mm, segment length=1mm}] (\a:25mm) arc (\a:\b:25mm);
  }

  \foreach \a / \b in {115/155, 218/232, 308/322, 38/52}
  {
    \draw [thick] (\a:25mm) arc (\a:\b:25mm);
  }

  \foreach \a / \b in {-38/52, 308/155, 218/38, 232/115}
  {
    \draw (\a:25mm) -- (\b:25mm);
  }
  
\end{tikzpicture}
  \end{tabular}
  \caption{Two $l$-shortcuts of types (i) and (ii), respectively; marked are the cycles of length $n-l$.}
  \label{fig:shortcuts}
\end{figure}

Our strategy will be to show that a.a.s.~every subgraph $G'$ of $G(n,p)$ with more than $(\frac{1}{2} + \varepsilon)\frac{n^2p}{2}$ edges contains an $l$-shortcut for each $l \in \{0, \ldots, \delta n\}$ and every possible labeling of the vertices of $G'$. We will first prove it when $G(n,p)$ above is replaced by the complete graph ($p=1$), and extend it to all range of $p$ by applying Theorem~\ref{thm_transference}. 

\begin{LEMMA}
  \label{lemma_saturation}
  For every positive $\varepsilon_0$, there exists an integer $n_0$ such that if $\varepsilon' \ge \varepsilon_0$ and $n \ge n_0$, then every $n$-vertex graph $G'$ with $e(G') \geq (\frac{1}{2} + \varepsilon')\frac{n^2-n}{2}$ contains at least $(\frac{\varepsilon'}{16})^8 n^4$ $l$-shortcuts for every $l \in \{0, \ldots, \frac{\varepsilon'}{16} n\}$ and every labeling of the vertex set of $G'$ with $\modulon$.
\end{LEMMA}
\begin{proof}
Let $n_0 = \frac{128}{\varepsilon_0^2}$ and assume that $\varepsilon_0 \le \varepsilon'$ and $n \ge n_0$. 
Fix a labeling of the vertex set of $G'$ with $\modulon$. For a $k \in \modulon$, let $\deg_{G'}(k)$ denote the degree in $G'$ of the vertex with the label $k$. By our assumption on $e(G')$ and $n$,
  \[
  \sum_{i = 0}^{\lfloor\frac{n}{2}-1\rfloor}\big( \deg_{G'}(2i) + \deg_{G'}(2i+1) \big) \geq 2e(G') - \deg_{G'}(n-1) \geq 2e(G') - n \geq (1 + \varepsilon')\frac{n^2}{2},
  \]
  and hence the set $I$ defined by
  \[
  I = \big\{ i \in \{0, \ldots, \lfloor\frac{n}{2}-1\rfloor\} \colon \deg_{G'}(2i) + \deg_{G'}(2i+1) \geq (1 + \frac{\varepsilon'}{2})n \big\}
  \]
  contains at least $\frac{\varepsilon'}{4}n$ elements. For a $k$ with $0 \leq k \leq (1-\frac{\varepsilon'}{4})n$, let $I(k) = \{i \in I \colon \deg_{G'}(2i) \in [k, k+\frac{\varepsilon'}{4}n)\}$. A simple averaging argument implies that there exists a $k$ such that
  \[
  |I(k)| \geq \frac{|I|}{\lceil\frac{4}{\varepsilon'}\rceil} \geq \frac{\varepsilon'}{8}|I| \geq \frac{(\varepsilon')^2}{32}n.
  \]
  Fix any such $k$ and let $I' = I(k)$. Observe that for all $i, j \in I'$, we have $|\deg_{G'}(2i) - \deg_{G'}(2j)| \leq \frac{\varepsilon'}{4}n$, and hence
  \begin{equation}
    \label{eq:deg2i2jp}
    \deg_{G'}(2i) + \deg_{G'}(2j+1) \geq \deg_{G'}(2j) + \deg_{G'}(2j+1) - \frac{\varepsilon'}{4}n \geq \left(1 + \frac{\varepsilon'}{4}\right)n.
  \end{equation}
Another simple averaging argument implies that there is a subset $I'' \subset I'$ with at least $\frac{\varepsilon'}{16} |I'|$ elements, such that for all $i,j \in I''$, the distance between $2i$ and $2j$ on the cycle $C_n$ satisfies $\|2i - 2j\| \leq \frac{\varepsilon'}{16}n$. The assertion of Lemma~\ref{lemma_saturation} now easily follows from the following claim.

\begin{CLAIM}
  \label{claim_shortcuts}
  For every pair $i, j$ of distinct elements of $I''$, and every $l$ with $0 \leq l \leq \frac{\varepsilon'}{16} n$, there are at least $\left(\frac{\varepsilon'}{16}n\right)^2$ $l$-shortcuts with $\{i_1, i_2\} = \{2i, 2j\}$.
\end{CLAIM}
\begin{proof}[Proof of Claim~\ref{claim_shortcuts}]
Fix an $l$ with $0 \le l \le \frac{\varepsilon'}{16} n$. By the definition of $I''$, $\|2i - 2j\| \leq \frac{\varepsilon'}{16}n$. Let $\{i_1, i_2\} = \{2i, 2j\}$, such that $i_2 = i_1 + k$, for some $0 < k \leq \frac{\varepsilon'}{16}n$, where the addition is modulo $n$, i.e., $i_2$ lies in distance at most $\frac{\varepsilon'}{16}n$ from $i_1$ while moving around $C_n$ in the clockwise direction. Denote by $A$ the set $\{i_2 + 2, \ldots, i_1 - 1\}$ of vertices of $C_n$ lying on the longer of the two arcs connecting $i_2 + 1$ to $i_1$. Clearly, $|A| \geq \left(1-\frac{\varepsilon'}{16}\right)n - 2$. Let $A' = \{i \in A \colon i+l+1 \in A\}$ and note that $|A'| \geq |A| - (l+1) \geq (1-2\cdot\frac{\varepsilon'}{16}-\frac{\varepsilon'}{16})n$. Let
  \[
  B = \big\{i \in A' \colon \{i_1 + 1, i\}, \{i_2, i+l+1\} \in E(G')\big\}.
  \]
  Moreover, let $N_1 = N_{G'}(i_1 + 1)$ and $N_2 = \{i \in \modulon \colon i + l + 1 \in N_{G'}(i_2)\}$, and note that $B = A' \cap N_1 \cap N_2$. By \eqref{eq:deg2i2jp},
  \begin{eqnarray}
    \label{ineq:B}
    \left( 1 + \frac{\varepsilon'}{4} \right)n & \leq & \deg_{G'}(i_1 + 1) + \deg_{G'}(i_2) = |N_1| + |N_2| = |N_1 \cup N_2| + |N_1 \cap N_2| \\
    \nonumber
    & = & |N_1 \cup N_2| + |(\modulon - A') \cap N_1 \cap N_2| + |A' \cap N_1 \cap N_2| \leq 2n - |A'| + |B|,
  \end{eqnarray}
  and therefore $|B| \geq (\frac{\varepsilon'}{4} - 2\cdot\frac{\varepsilon'}{16} - \frac{\varepsilon'}{16})n \geq \frac{\varepsilon'}{16}n$. Fix some $i_4 \in B$, let $J = \{i_4, \ldots, i_4+l+1\}$, $A'' = \{i \in A \setminus J \colon i+1 \in A \setminus J\}$, and note that $|A''| \geq |A| - |J| - 2 \geq (1 - 2\cdot\frac{\varepsilon'}{16} - \frac{\varepsilon'}{16})n$. Moreover, let
  \[
  C = \big\{i \in A'' \colon \{i_1, i\}, \{i_2 + 1, i + 1\} \in E(G') \big\}.
  \]
  Similarly as in \eqref{ineq:B},
  \[
  \left( 1 + \frac{\varepsilon'}{4} \right)n \leq \deg_{G'}(i_1) + \deg_{G'}(i_2+1) \leq 2n - |A''| + |C|,
  \]
  and therefore $|C| \geq (\frac{\varepsilon'}{4} - 2\cdot\frac{\varepsilon'}{16} - \frac{\varepsilon'}{16})n \geq \frac{\varepsilon'}{16}n$. Finally, note that for every $i_3 \in C$, the graph $X \subset G'$ consisting of the edges $\{i_1, i_3\}, \{i_1+1, i_4\}, \{i_2, i_4 + l + 1\}$, and $\{i_2 + 1, i_3 + 1\}$ is an $l$-shortcut.
\end{proof}

To finish the proof of Lemma~\ref{lemma_saturation}, note that by Claim~\ref{claim_shortcuts}, the total number of $l$-shortcuts is at least
\[
{|I''| \choose 2} \cdot \left(\frac{\varepsilon'}{16}n\right)^2 \geq \frac{|I''|^2}{4} \cdot \left(\frac{\varepsilon'}{16}\right)^2 \cdot n^2 \geq \frac{|I'|^2}{4} \cdot \left(\frac{\varepsilon'}{16}\right)^4 \cdot n^2 \geq \left(\frac{\varepsilon'}{16}\right)^8 \cdot n^4.
\]
\end{proof}

Given integers $l$ and $n$ with $0 \leq l \leq \frac{n}{2}$, let $H_n^l$ be the $4$-uniform hypergraph on the vertex set $V(H_n^l) = E(K_n)$ whose hyperedges are all $l$-shortcuts in $K_n \setminus C_n$. More precisely, every hyperedge corresponds to some four edges in the graph (four vertices in the hypergraph) that form an $l$-shortcut. Since the sequence $(i_1, i_2, i_3, i_4)$ uniquely determines an $l$-shortcut, one can observe that
\begin{equation}
  \label{ineq:EHln}
  |V(H^l_n)| = |E(K_n)| = \frac{n^2-n}{2} \quad\text{and}\quad cn^4 \leq |E(H^l_n)| \leq n^4,
\end{equation}
where $c$ is some positive constant that does not depend on $l$ or $n$. In order to apply Theorem~\ref{thm_transference} to the hypergraphs $H^l_n$ and find $l$-shortcuts in subgraphs of random graphs, 
we need to show that they are well-behaved, i.e., $H_n^l$ is $(\alpha, f, \varepsilon_0)$-dense (see Definition~\ref{dfn_alphadense}) and $(K,(n^{-1/2}))$-bounded (see Definition~\ref{dfn_Kpbounded}) for some suitably chosen $\alpha$, $f$, $\varepsilon_0$, and $K$. 

The following statement is an immediate corollary of Lemma \ref{lemma_saturation}.

\begin{COR}
  \label{cor_shortcuts}
  Let $f \colon (0, 1) \rightarrow (0,1)$ be the (non-decreasing) function defined by $f(\varepsilon') = (\frac{\varepsilon'}{16})^8$ for all $\varepsilon' \in (0,1)$. For every positive $\varepsilon_0$, there exist a positive $\delta$ and an integer $n_0$ such that for all $l$ and $n$ satisfying $n \ge n_0$ and $0 \leq l \leq \delta n$, the hypergraph $H^l_n$ is $(\frac{1}{2}, f, \varepsilon_0)$-dense.
\end{COR}

Our next lemma establishes $(K,(n^{-1/2}))$-boundedness of the hypergraphs $H_n^l$.

\begin{LEMMA}
  \label{lemma_Kpbounded}
  There exists a positive constant $K$ such that for all $n$ and $l$ with $0 \leq 2l \leq n$, the hypergraph $H^l_n$ is $(K, (n^{-1/2}))$-bounded.
\end{LEMMA}
\begin{proof}
  First, observe that once we fix $i$ edges of an $l$-shortcut for some $i \in \{1,2,3\}$, we are fixing at least $i+1$ elements in the sequence $(i_1, i_2, i_3, i_4)$ that uniquely describes this $l$-shortcut (up to at most 8! possibilities as each endpoint of an edge can correspond to $i_1, i_1+1, i_2,$ etc.). Therefore, there are absolute constants $c_1$, $c_2$, and $c_3$, such that every $i$ edges of $K_n \setminus C_n$ are contained in at most $c_in^{3-i}$ $l$-shortcuts for every $i \in \{1, 2, 3\}$ and all $n$ and $l$ with $n \geq 2l$ (we require $n \geq 2l$ in order to have enough space between $i_4$ and $i_4 + l + 1$). 

 Assume that $q \geq n^{-1/2}$, and fix integers $n$ and $l$ with $n \geq 2l$. Let $e$ be an edge of $K_n \setminus C_n$ (hence $e$ is also a vertex of $H^l_n$). For $i \in \{1, 2, 3\}$, define $d_i(e)$ as,
  \[
  d_i(e) = \Ex\left[ \deg_i(e,V(H^l_n)_q)^2 \right],
  \]
  where $\deg_i(e,V(H^l_n)_q) = |\{X \in E(H^l_n) \colon |X \cap (V(H^l_n)_q \setminus \{ e \} )| \ge i \text{ and } e \in X \}|$ is defined before Definition~\ref{dfn_Kpbounded}. Observe that $\deg_i(e,V(H^l_n)_q)^2$ is equal to
  \[
  |\{(X_1, X_2) \in E(H^l_n)^2 \colon |X_k \cap (V(H^l_n)_q \setminus \{ e \} )| \ge i \text{ for } k = 1,2 \text{ and } e \in X_1 \cap X_2 \}|.
  \]
  Thus, in order to calculate $d_i(e)$ we can consider all pairs of $l$-shortcuts $X_1, X_2$ that share the edge $e$, and consider the possible combinations of $|(X_1 \cap X_2) \cap V(H^l_n)_q|$, $|X_1 \cap V(H^l_n)_q|$, and $|X_2 \cap V(H^l_n)_q|$ that give $|X_k \cap (V(H^l_n)_q \setminus \{ e \} )| \ge i$ for $k=1,2$. By parametrizing the
size of $\abs{(X_1 \cap X_2) \cap (V(H^l_n)_q \setminus \{ e \} )}$ by $j$, this interpretation gives the inequality
  \begin{equation}
    \label{ineq_die}
    d_i(e) \le \sum_{e \in X_1 \cap X_2} \sum_{j=0}^{\min\{i,|X_1 \cap X_2| - 1\}} {4 \choose i}^2 q^{2i-j}.
  \end{equation}
  For $t \in \{1, 2, 3, 4\}$, let ${\cal F}_t = |\{(X_1, X_2) \colon |X_1 \cap X_2| \geq t, e \in X_1 \cap X_2\}|$. Once we bound the size of the sets ${\cal F}_t$, we can obtain a bound on $d_i(e)$ using inequality~\eqref{ineq_die}. As mentioned at the beginning of the proof, if we fix $t$ edges of an $l$-shortcut, then we are fixing at least $t+1$ elements of the sequence $(i_1, i_2, i_3, i_4)$ describing an $l$-shortcut (up to at most $8!$ possibilities). Since both $X_1$ and $X_2$ contain the edge $e$, at least two values are fixed for both $X_1$ and $X_2$. Thus, $|\mathcal{F}_1| = O(n^{4})$. Similarly, we have $|\mathcal{F}_t| = O(n^{5-t})$ for $t \in \{1, 2, 3\}$ and  $|\mathcal{F}_4| = O(n^{2})$. It follows that
  \begin{align*}
  d_1(e) & \leq 4^2 \cdot \left( |\mathcal{F}_1|q^2 + |\mathcal{F}_2|q \right) = O(n^4q^2 + n^3q) \leq c' n^4 q^2,\\
  d_2(e) & \leq 6^2 \cdot \left( |\mathcal{F}_1|q^4 + |\mathcal{F}_2|q^3 + |\mathcal{F}_3|q^2 \right) = O(n^4q^4 + n^3q^3 + n^2q^2) \leq c' n^4 q^4,\\
  d_3(e) & \leq 4^2 \cdot \left( |\mathcal{F}_1|q^6 + |\mathcal{F}_2|q^5 + |\mathcal{F}_3|q^4 + |\mathcal{F}_4|q^3 \right) = O(n^4q^6 + n^3q^5 + n^2q^4 + n^2q^3) \leq c' n^4 q^6
\end{align*}
for some absolute constant $c'$, where the last inequality in each row holds by our assumption that $q \ge n^{-1/2}$. Finally, recall from~\eqref{ineq:EHln} that $|E(H^l_n)|^2/|V(H^l_n)|^2 \geq cn^4$ for some positive constant $c$ not depending on $n$ or $l$. It follows that for each $i \in \{1, 2, 3\}$,
\[
\Ex\left[ \sum_{e \in V(H^l_n)} \deg_i(e,V(H^l_n)_q)^2 \right] = \sum_{e \in V(H^l_n)} \Ex\left[ \deg_i(e,V(H^l_n)_q)^2 \right] \leq \frac{c'}{c} \cdot q^{2i} \frac{|E(H^l_n)|^2}{|V(H^l_n)|},
\]
i.e., the sequence $(H^l_n)$ is $(\frac{c'}{c}, (n^{-1/2}))$-bounded.
\end{proof}

Given a positive real $\delta$, let $\Pd$ be the following graph property: An $n$-vertex graph $G$ satisfies $\Pd$ if and only if $G$ contains an $l$-shortcut for every $l$ with $0 \leq l \leq \delta n$ and every labeling of the vertices of $G$ with $\modulon$. Clearly, $\Pd$ is monotone increasing.

\begin{LEMMA}
  \label{lemma_Pd}
  For all positive $\varepsilon$, there exist positive $\delta$ and $C$ such that if $Cn^{-1/2} \leq p(n) \leq n^{-1/3}$, then $G(n,p)$ a.a.s.~satisfies the following. Every subgraph $G'$ of $G(n,p)$ with more than $(\frac{1}{2} + \varepsilon)e(G(n,p))$ edges satisfies $\Pd$.
\end{LEMMA}
\begin{proof}
 Let $k = 4$, $\alpha = \frac{1}{2}$, $p_n = n^{-1/2}$, $v_n = {n \choose 2}$, $e_n = cn^4$, where $c$ is a positive constant that satisfies~\eqref{ineq:EHln}, and let $K$ be as in the statement of Lemma~\ref{lemma_Kpbounded}.  Let $f$ be the function defined in Corollary~\ref{lemma_Kpbounded}.
  Note that
  \[
  \lim_{n \to \infty} e_n(n^{-1/2})^4 = \infty \quad \text{and} \quad \lim_{n \to \infty} v_nn^{-1/2} = \infty,
  \]
  and let $b$, $\varepsilon_0$, $C$, and $n_0$ be the numbers satisfying the conclusion of Theorem~\ref{thm_transference} for $\alpha$, $f$, $\varepsilon$, $k$, $K$, $(p_n)$, $(v_n)$, and $(e_n)$ as above. Let $n_1$ and $\delta$ be the parameters obtained from Corollary \ref{cor_shortcuts} by using the parameter $\varepsilon_0$. 
  
  Assume that $n \geq \max\{n_1, n_0\}$ and fix an $l$ with $0 \leq l \leq \delta n$. Furthermore, fix a labeling of the vertices of $G(n,p)$ with $\modulon$. By Corollary~\ref{cor_shortcuts}, the hypergraph $H^l_n$ is $(\frac{1}{2}, f, \varepsilon_0)$-dense, and by Lemma~\ref{lemma_Kpbounded}, it is $(K, (n^{-1/2}))$-bounded. Moreover, $|V(H^l_n)| \geq v_n$ and $|E(H^l_n)| \geq e_n$. Let $p$ satisfy $Cn^{-1/2} \leq p(n) \leq n^{-1/3}$. Since $n \geq n_0$, it follows from Theorem~\ref{thm_transference} that with probability at least $1 - \exp(-bp{n \choose 2})$, every $G' \subset G(n,p)$ with $e(G') \geq (\frac{1}{2}+\varepsilon)e(G(n,p))$ contains a hyperedge of $H^l_n$, which is an $l$-shortcut with respect to our fixed labeling of the vertices. Hence, with probability at least $1 - n! \cdot n \cdot \exp(-bp{n \choose 2})$, which clearly is $1 - o(1)$, the random graph $G(n,p)$ satisfies $\Pd$.
\end{proof}

Having proved Lemma~\ref{lemma_Pd}, we are finally ready to show the existence of short and long cycles.

\begin{proof}[Proof of Theorem~\ref{thm_mainthmpart3}]
  As it was remarked at the beginning of this section, we can restrict our attention to cycles of length at least $8$. Let $C$ and $\delta$ be numbers satisfying the conclusion of Lemma~\ref{lemma_Pd} for $\frac{\varepsilon}{4}$. Moreover, let $p'(n) = Cn^{-1/2}$. Since a.a.s.~$e(G(n,p')) \leq (1+\frac{\varepsilon}{8})\frac{n^2p'}{2}$, then by Lemma~\ref{lemma_Pd}, a.a.s.~every subgraph of $G(n,p')$ with more than $(\frac{1}{2} + \frac{\varepsilon}{2})\frac{n^2p'}{2}$ edges satisfies $\Pd$. By Proposition~\ref{prop_probrestriction}, if $p(n) \geq Cn^{-1/2}$, then a.a.s.~every subgraph $G'$ of $G(n,p)$ with at least $(\frac{1}{2}+\varepsilon)\frac{n^2p}{2}$ edges satisfies $\Pd$. Finally, observe that every Hamiltonian graph with property $\Pd$ contains a cycle of length $t$ for every $t \in [8, \delta n] \cup [(1-\delta)n,n]$.
\end{proof}


\subsection{Medium Length Cycles}

\begin{THM} \label{thm_mainthmpart2}
  For all positive $\varepsilon$ and $\delta$, there exists a constant $C$ such that if $p \geq Cn^{-1/2}$, then $G(n,p)$ a.a.s.~satisfies the following: Every Hamiltonian subgraph $G'$ of $G(n,p)$ with more than $(\frac{1}{2}+\varepsilon)\frac{n^2p}{2}$ edges contains a cycle of length $t$ for all $t \in [\delta n, (1-\delta)n]$.
\end{THM}

Fix a labeling of the vertices of the complete graph $K_n$ with numbers from the set $\modulon$ of remainders modulo $n$. Throughout this section, the addition of the elements of $\modulon$ will be performed modulo $n$. We define the following partition of $E(K_n)$:
\[
E(K_n) = \bigcup_{i = 0}^{n-1}E_i, \mbox{ where } E_i = \big\{\{x,y\} \colon x+y \equiv i \mod n\big\}.
\]
We will later refer to the elements of $E_i$ as the edges in {\em direction} $i$. For each $i \in \modulon$, there is a natural ordering $\le_i$ of the elements of $E_i$ defined as follows. If we evenly place the numbers from $\modulon$ on a circle, each set $E_i$ will consist of all parallel edges in certain direction. We order these edges with respect to their distance from the point $\frac{i}{2}$ (the arc from $\lfloor \frac{i}{2} \rfloor$ to $\lceil \frac{i}{2} \rceil$ in case $i$ is odd), see Figure~\ref{fig:Ei}. For example, $E_0$ consists
of the $\lfloor \frac{n}{2} \rfloor$ edges $\{x, -x\}$ for positive integers $x < \frac{n}{2}$, and $\{x_1, -x_1\} \le_0 \{x_2, -x_2\}$ for two positive integers $x_1, x_2 < \frac{n}{2}$ if and only if $x_1 \le x_2$.
Since every $n$-vertex graph $G$ is a subgraph of the complete graph $K_n$, we tacitly assume that when we fix a labeling of the set of vertices of $G$ with $\modulon$, $E(G)$ inherits the partition into sets $E_i$ with their orders $\leq_i$.

Let $C_n$ be the subgraph of $K_n$ consisting of the edges $\{i,i+1\}$ for all $i \in \modulon$. We call two edges $e_1, e_2 \in E(K_n) \setminus E(C_n)$ {\em crossing} if their endpoints are all distinct and lie alternately on the cycle $C_n$. With all this notation at hand, observe that for every $i \in \modulon$ and $l \in \{2, \ldots, n-2\}$, the graph $C_n \cup \{e_1,e_2\}$, where $e_1 \in E_i$ and $e_2 \in E_{i+l}$ are crossing edges, contains cycles of lengths $l+2$ and $n-l+2$, see Figure~\ref{fig:Ei}.

\begin{figure}[h]
  \centering
  \begin{tabular}{ccc}

\begin{tikzpicture}
  \clip (-2.8,-3.0) rectangle (3.5,3.2);

  \draw [thick] (0,0) circle (25mm);

  \draw [fill=black] (45:25mm) circle (0.5mm);
  \draw (45:29mm) node {\footnotesize{$\frac{i}{2}$}};

  \draw [fill=black] (225:25mm) circle (0.5mm);
  \draw (225:30mm) node {\footnotesize{$\frac{n+i}{2}$}};

  \foreach \a / \i in {25/1,45/2,65/3}
  {
    \draw (45-\a:25mm) -- (45+\a:25mm);
    \draw [fill=black] (45-\a:25mm) circle (0.5mm);
    \draw [fill=black] (45+\a:25mm) circle (0.5mm);
    \draw (45-\a:30mm) node {\footnotesize{$\frac{i}{2}+\i$}};
    \draw (45+\a:29mm) node {\footnotesize{$\frac{i}{2}-\i$}};
    
    \draw (225-\a:25mm) -- (225+\a:25mm);
    \draw [fill=black] (225-\a:25mm) circle (0.5mm);
    \draw [fill=black] (225+\a:25mm) circle (0.5mm);
  }

  \draw [thick,->] (0,1) -- node [above,xshift=-2mm] {\small{$\leq_i$}} (-1,0);

  \foreach \x in {-0.25,0,0.25}
  {
    \draw [fill=black] (\x,\x) circle (0.2mm);
  }
\end{tikzpicture} & \hspace{0.5in} & 

\begin{tikzpicture}
  \clip (-2.8,-3.0) rectangle (3.5,3.2);

  \draw [fill=black] (120:25mm) circle (0.5mm);
  \draw (120:28mm) node {\footnotesize{$x$}};
  \draw [fill=black] (-50:25mm) circle (0.5mm);
  \draw (-50:29mm) node {\footnotesize{$i-x$}};
  \draw (120:25mm) -- (-50:25mm);

  \draw [fill=black] (45:25mm) circle (0.5mm);
  \draw (45:28mm) node {\footnotesize{$y$}};
  \draw [fill=black] (-100:25mm) circle (0.5mm);
  \draw (-100:28mm) node {\footnotesize{$i+l-y$}};
  \draw (45:25mm) -- (-100:25mm);

  \draw [thick, decorate, decoration={snake, amplitude=.2mm, segment length=1mm}] (45:25mm) arc (45:120:25mm);
  \draw [thick, decorate, decoration={snake, amplitude=.2mm, segment length=1mm}] (-50:25mm) arc (-50:-100:25mm);
  \draw [thick] (120:25mm) arc (120:260:25mm);
  \draw [thick] (-50:25mm) arc (-50:45:25mm);
\end{tikzpicture}
  \end{tabular}
  \caption{The set $E_i$ in the case when both $i$ and $n$ are even, the arrow points from $\leq_i$-smaller to $\leq_i$-larger elements; a crossing between an edge $\{x,i-x\} \in E_i$ and an edge $\{y, i+l-y\} \in E_{i+l}$, marked is the cycle of length $l+2$.}
    \label{fig:Ei}
\end{figure}
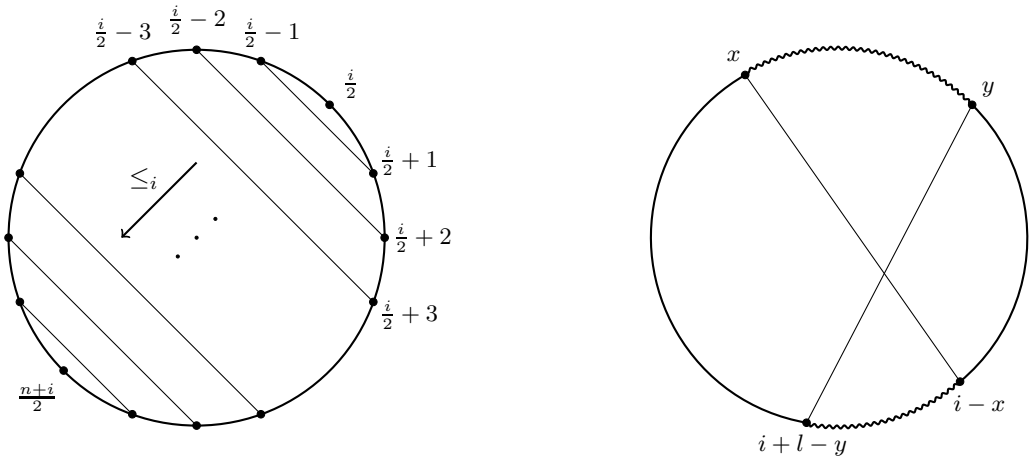

Our strategy will be to show that a.a.s.~every subgraph $G'$ of $G(n,p)$ with more than $(\frac{1}{2} + \varepsilon)\frac{n^2p}{2}$ edges contains such pair $e_1, e_2$ for each $l \in \{\delta n - 2, \ldots, (1-\delta)n \}$ and every possible labeling of the vertices of $G'$. Before we start working with the random graph, we need to introduce some more notation.

Note that even though the size of $E_i$ depends on the parity of $i$ and $n$, it always satisfies $\frac{n}{2}-1 \leq |E_i| \leq \frac{n}{2}$. Given a $\beta \in (0, \frac{1}{6})$ and a natural number $k$ with $1 \leq k \leq (\frac{1}{2} - \beta)n$, denote by $E_i^k$ the set of $\beta n$ consecutive (with respect to $\le_i$) edges in $E_i$, starting from the $k$\textsuperscript{th} smallest, i.e., $E_i^k$ is the interval of length $\beta n$ in $E_i$ whose leftmost endpoint is the $k$\textsuperscript{th} smallest element of $E_i$. Also, let $M_i \subset E_i$ denote the set of $(\frac{1}{2}-2\beta)n$ middle elements of $E_i$ in the
ordering $\le_i$, i.e., $E_i$ without the leftmost and rightmost subintervals of lengths $\beta n$. Let $G$ be a graph on the vertex set $\modulon$. For $i \in \modulon$, $\varepsilon' > 0$ and $p \in [0,1]$, we will say that the direction $E_i$ is {\em $(\beta, \varepsilon', p)$-good in $G$} (or simply {\em good}) if for all $k$ with $1 \leq k \leq (\frac{1}{2} - \beta)n$, $G$ satisfies
\[
\big||E(G) \cap E_i^k| - \beta n p\big| \leq \varepsilon' \beta n p \quad \text{and} \quad \left||E(G) \cap M_i| - \left(\frac{1}{2}-2\beta\right) n p\right| \leq \varepsilon' (\frac{1}{2} - 2\beta) n p.
\]

\begin{LEMMA}
  \label{lemma_gooddirections}
  If $p \geq Cn^{-1/2}$ for some positive constant $C$, and $\beta, \varepsilon' \in (0,\frac{1}{6})$, then a.a.s.~for every labeling of the vertices of $G(n,p)$ with $\modulon$, there are at most $n^{3/4}$ directions that are not $(\beta, \varepsilon', p)$-good in $G(n,p)$.
\end{LEMMA}
\begin{proof}
  Let $G$ be a graph drawn from $G(n,p)$ and fix a labeling of the set of vertices of $G$ with $\modulon$. By Chernoff's inequality (Theorem~\ref{thm_Chernoff}) , for every $i$ and $k$,
  \[
  P\left(\big||E(G) \cap E_i^k| - \beta n p\big| > \varepsilon' \beta n p\right) \leq e^{-cnp}
  \]
  for some positive constant $c = c(\beta,\varepsilon')$, and similarly (note that $\frac{1}{2} - 2\beta \geq \beta$)
  \[
  P\left(\left||E(G) \cap M_i| - \left(\frac{1}{2}-2\beta\right) n p\right| > \varepsilon' \left(\frac{1}{2}-2\beta\right) n p\right) \leq e^{-cnp}.
  \]
  Hence, the probability that a particular $E_i$ is not good in $G$ is at most $e^{-cnp/2}$ by the union bound. Since the sets $E_i$ are pairwise disjoint, the events `$E_i$ is not good' are mutually independent. Therefore, the probability that there are more than $n^{3/4}$ bad directions is at most
  \[
  {n \choose n^{3/4}}\left(e^{-cnp/2}\right)^{n^{3/4}} \leq 2^n \cdot e^{-cn^{7/4}p/2} = e^{-c' n^{5/4}},
  \]
for some constant $c'$ depending on $c$ and $C$. 
Finally, since there are only $n!$ different labellings, the probability of there
being a labeling with more than $n^{3/4}$ bad directions is $o(1)$.
\end{proof}

Recall that, given a labeling of the vertex set of $K_n$ with $\modulon$, we say that two edges $e_1, e_2 \in E(K_n) \setminus E(C_n)$ cross if their endpoints lie alternately on the cycle $C_n$. Since in the proof of Theorem~\ref{thm_mainthmpart2}, we are planning to use an averaging argument to show the existence of a suitable pair of crossing edges, it would be very convenient if for every $i$ and $l$, every edge from $E_i$ crossed the same number of edges from $E_{i+l}$. Unfortunately, this is not true. To overcome this obstacle, we introduce the notion of close crossings, for which such uniformity statement holds approximately. Similarly as before, for each $k \in \modulon$, denote its distance from $0$ on the cycle $C_n$ by $\|k\|$; more precisely, let
\[
\| k \| = \min \left\{ x \geq 0 \colon x \equiv k \mod n \text{ or } x \equiv -k \mod n \right\}.
\]
Given a $\beta \in (0,\frac{1}{6})$, a crossing between edges $\{x_1,y_1\}$ and $\{x_2,y_2\}$ will be called {\em close} if the smallest of the distances between their endpoints is at most $\beta n$, i.e., $\min\{\|x_1-x_2\|, \|x_1-y_2\|, \|y_1-x_2\|, \|y_1-y_2\|\} \leq \beta n$. The following simple observations will be crucial in the proof of Theorem~\ref{thm_mainthmpart2}.

\begin{LEMMA}
  \label{lemma_closecrossings}
  For every $i \in \modulon$, $\beta \in (0,\frac{1}{6})$, and $l \in \{2\beta n + 1, \ldots, (1-2\beta)n-1\}$, the following holds.
  \begin{enumerate}
    \setlength{\itemsep}{1pt}
    \setlength{\parskip}{0pt}
    \setlength{\parsep}{0pt}
  \item[(i)]
    Every edge in $E_i$ forms close crossings with at most $2\beta n$ edges from $E_{i+l}$, and these edges can be covered by a set of the form $E_{i+l}^{k_1} \cup E_{i+l}^{k_2}$ for some $k_1$ and $k_2$ with $1 \leq k_1, k_2 \leq (\frac{1}{2}-\beta)n$.
  \item[(ii)]
    At least $(\frac{1}{2} - 2\beta)n$ edges in $E_i$ form close crossings with exactly $2\beta n$ edges from $E_{i+l}$, and these $2\beta n$ edges constitute a set of the form $E_{i+l}^{k_1} \cup E_{i+l}^{k_2}$ for some $k_1$ and $k_2$ with $1 \leq k_1, k_2 \leq (\frac{1}{2}-\beta)n$.
  \item[(iii)]
    The $(\frac{1}{2} - 2\beta)n$ edges from (ii) cover the interval $M_i \subset E_i$.
  \end{enumerate}
\end{LEMMA}
\begin{proof}
  Fix an edge $e_1 \in E_i$ and assume that $e_1 = \{x_1, y_1\}$. Note that the set $C(e_1)$ of edges in $E_{i+l}$ that cross $e_1$ is an interval in the ordering $\leq_{i+l}$ of $E_{i+l}$. The edges in $E_{i+l}$ that form close crossings with $e_1$ come precisely from the leftmost and the rightmost subintervals of $C(e_1)$ of lengths $\beta n$.
Hence $(i)$ follows. Moreover, the length of $C(e_1)$ is $\min\{\|x_1-y_1\|, \|l\|, n - \|l\| \} - 1$. Note that $\|x_1-y_1\| \leq 2\beta n$ precisely for those $2\beta n$ edges $\{x_1,y_1\} \in E_i$ which come from the leftmost and the rightmost subintervals of $E_i$ of lengths $\beta n$, i.e., $E_i \setminus M_i$. Also $\|l\| - 1,  n- \|l\| - 1 \geq 2 \beta n$ by our assumption on $l$, so $e_1$ forms a close crossing with less than $2\beta n$ edges if and only if $\|x_1-y_1\| \leq 2\beta n$. This proves $(ii)$ and $(iii)$.
\end{proof}

With the above two lemmas at hand, we are ready to show the existence of medium length cycles.

\begin{proof}[Proof of Theorem \ref{thm_mainthmpart2}]
  Let $\beta = \min\{\frac{\delta}{4}, \frac{\varepsilon}{10}\}$ and $\varepsilon' = \frac{\varepsilon}{10}$. Let $G$ be a graph drawn from $G(n,p)$. By Lemma~\ref{lemma_gooddirections}, a.a.s.~every labeling of the vertices of $G$ with $\modulon$ yields at most $\varepsilon' n$ directions that are not $(\beta, \varepsilon', p)$-good in $G$. Moreover, by Chernoff's inequality, a.a.s.~$e(G) \leq (1+\frac{\varepsilon}{2})\frac{n^2p}{2}$. Fix a $t \in [\delta n, (1-\delta)n]$. We will show that, conditioned on the above two events (which hold a.a.s.), every Hamiltonian subgraph $G'$ of $G$ with more than $(\frac{1}{2}+\varepsilon)\frac{n^2p}{2}$ edges contains a cycle of length $t$.

  Fix such a subgraph $G'$ and a labeling of its vertices such that $C_n$ is a Hamilton cycle in $G'$, and let $l = t - 2$. It suffices to show that for some $i \in \modulon$, the graph $G'$ contains some $e_1 \in E_i$ and $e_2 \in E_{i+l}$ which form a close crossing, since we have already observed that the graph $C_n \cup \{e_1,e_2\}$ contains a cycle of length $l + 2$, and $l$ was chosen so that $l + 2 = t$. Let $I$ be the set of directions that are good in $G$ and recall that $|I| \geq (1-\varepsilon')n$. We start by estimating the number of close crossings between pairs of edges of $G$ which came from $E_i$ and $E_{i+l}$ such that both $i$ and $i+l$ are in $I$. Denote this quantity by $X$. For the remainder of the proof we will assume that $2l \neq n$; the case $2l = n$ can be resolved using an almost identical argument and we leave it to the reader. The number of pairs $\{i, i+l\} \subset I$ is at least $(1-2\varepsilon')n$ (here we use the assumption that $2l \neq n$ and hence $i + l \neq i-l$). Let us fix one of them. Since $l \in (3\beta n, (1-3\beta n))$ by our assumption on $t$, Lemma~\ref{lemma_closecrossings} proves that each of the at least $(1-\varepsilon')(\frac{1}{2}-2\beta)np$ edges in $M_i \cap E(G)$ forms a close crossing with every edge from some two disjoint sets $E_{i+l}^k$ of sizes $\beta n$ each. Since $i+l \in I$, the graph $G$ has at least $(1-\varepsilon')\beta n p$ edges in each such set $E_{i+l}^k$. It follows that
  \[
  X \geq (1-2\varepsilon')n \cdot (1 - \varepsilon')\left(\frac{1}{2} - 2\beta\right)np \cdot 2(1-\varepsilon')\beta n p \geq (1-4\varepsilon'-4\beta)\beta n^3p^2.
  \]
  On the other hand, since every edge $e_1 \in E_i$ forms close crossings with at most $2\beta n$ edges from $E_{i \pm l}$, and such edges are covered by sets $E_{i \pm l}^k$ (Lemma~\ref{lemma_closecrossings}), every edge in a good direction $i$ can form at most $(1+ \varepsilon') \cdot 4\beta n p$ close crossings with edges in a good direction $i \pm l$, i.e., crossings counted by $X$. Hence, the number $X'$ of crossings in $G$ that are counted by $X$ but are not contained in $G'$ satisfies
  \[
  X' \leq \big( e(G) - e(G') \big) \cdot (1+\varepsilon') \cdot 4\beta n p \leq (1+\varepsilon') \cdot \frac{1-\varepsilon}{2} \cdot \frac{n^2p}{2} \cdot 4\beta n p \leq (1 - \varepsilon + \varepsilon') \beta n^3 p^2.
  \]
  Since $5 \varepsilon' + 4\beta < \varepsilon$, it follows that $X > X'$ and hence $G'$ must contain a crossing pair of edges from some $E_i$ and $E_{i+l}$. Finally, recall that these two edges, together with the edges of the spanning cycle in $G'$, form a graph containing a cycle of length $t$.
\end{proof}


\section{Concluding Remarks}
\label{section_concludingremarks}

\subsection{Necessity of Hamiltonicity}

In this paper, we proved that if $p \gg n^{-1/2}$, then $G(n,p)$
a.a.s.~satisfies the following property: Every Hamiltonian subgraph
$G' \subset G(n,p)$ with more than $(\frac{1}{2} + o(1)){n \choose 2}p$ edges
is pancyclic. Our proof heavily relied on the fact that $G'$ contains
a Hamilton cycle, and as discussed in the introduction, it is important
to assume the Hamiltonicity. However, if we are only interested in finding short
cycles, cycles of length $3$ up to $\delta(\varepsilon) n$, then it is no longer
necessary to assume the Hamiltonicity of $G'$. The proof of this fact will be given in the appendix.

\subsection{4-pancyclicity}

The general idea of our proof of Theorem \ref{thm_mainthm}
was to find a particular configuration of edges
which form a cycle of desired length together with the given Hamilton cycle.
Let $k$-crossing be a configuration consisting of $k$ edges depicted in Figure~\ref{fig:k-crossing}.
When finding medium length cycles, we only considered 2-crossings (which we called crossing pairs),
while when finding short and long cycles, we considered 4-crossings (which we called shortcuts).
One of the reasons that the proof consists of two separate parts
is because it is impossible to use only 2-crossings to find long cycles.
For example, after fixing a Hamilton cycle, there are only about $n^2$ 2-crossings in the
complete graph $K_n$ that form a cycle of length $n-1$ together with the fixed Hamilton cycle;
only about $n^2p^2$ of them will appear in $G(n,p)$. Since
$n^2p^2 \ll n^2p$, we can easily remove all the 2-crossings which
form a cycle of length $n-1$ together with the given Hamilton cycle.
This issue has been resolved by considering 4-crossings.

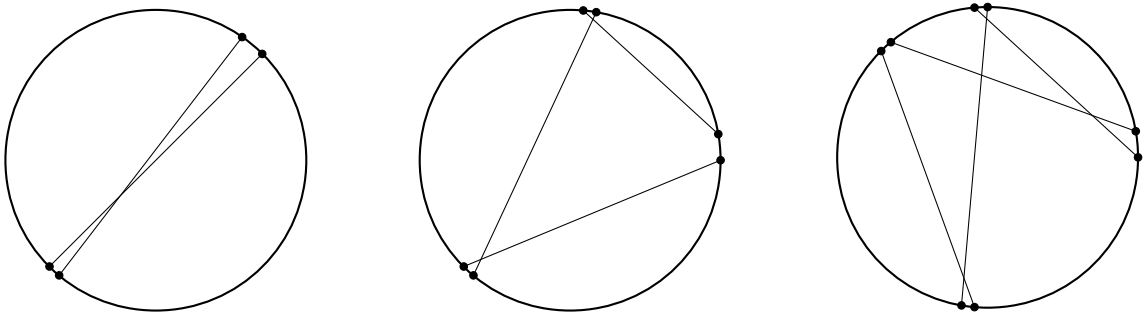
\begin{figure}[h!]
  \centering
  \begin{tabular}{ccccc}

\begin{tikzpicture}

  \draw [thick] (0,0) circle (20mm);

  \draw [fill=black] (45:20mm) circle (0.5mm);

  \draw [fill=black] (55:20mm) circle (0.5mm);

  \draw [fill=black] (225:20mm) circle (0.5mm);

  \draw [fill=black] (230:20mm) circle (0.5mm);

  \draw (45:20mm) -- (225:20mm);
    
  \draw (55:20mm) -- (230:20mm);

\end{tikzpicture} & \hspace{0.2in} & 

\begin{tikzpicture}

  \draw [thick] (0,0) circle (20mm);

  \draw [fill=black] (0:20mm) circle (0.5mm);

  \draw [fill=black] (10:20mm) circle (0.5mm);

  \draw [fill=black] (80:20mm) circle (0.5mm);
  
  \draw [fill=black] (85:20mm) circle (0.5mm);

  \draw [fill=black] (225:20mm) circle (0.5mm);

  \draw [fill=black] (230:20mm) circle (0.5mm);

  \draw (0:20mm) -- (225:20mm);
    
  \draw (85:20mm) -- (10:20mm);
  
  \draw (230:20mm) -- (80:20mm);

\end{tikzpicture} & \hspace{0.2in} & 

\begin{tikzpicture}

  \draw [thick] (0,0) circle (20mm);

  \draw [fill=black] (0:20mm) circle (0.5mm);

  \draw [fill=black] (10:20mm) circle (0.5mm);

  \draw [fill=black] (90:20mm) circle (0.5mm);

  \draw [fill=black] (95:20mm) circle (0.5mm);

  \draw [fill=black] (130:20mm) circle (0.5mm);

  \draw [fill=black] (135:20mm) circle (0.5mm);

  \draw [fill=black] (260:20mm) circle (0.5mm);

  \draw [fill=black] (265:20mm) circle (0.5mm);

  \draw (130:20mm) -- (10:20mm);
    
  \draw (135:20mm) -- (265:20mm);

  \draw (95:20mm) -- (0:20mm);
  
  \draw (90:20mm) -- (260:20mm);

\end{tikzpicture}
  \end{tabular}
  \caption{$2,3$ and $4$-crossings.}
  \label{fig:k-crossing}
\end{figure}

Following Cooper~\cite{Cooper1}, we call a graph {\em $k$-pancyclic} if
there is some Hamilton cycle in the graph such that for all $t \in \{3, \ldots, n\}$,
we can find a cycle of length $t$ using only the edges of this fixed Hamilton cycle
and at most $k$ other edges. Using this terminology, Theorems~\ref{thm_mainthmpart3}
and~\ref{thm_mainthmpart2} say that a.a.s.~every subgraph $G'$ as above
is `almost' $4$-pancyclic (recall that the existence of cycles of length 3 to 7
relied on Theorems~\ref{thm_resilientsmallcycle} and \ref{thm_bondysimonovits}). 
In fact, also for $t$ with $3 \le t \le 7$, we can show that there exists a cycle
of length $t$ which uses only the edges of the given Hamilton cycle and at most 4 other edges.
This statement is vacuously true for cycles of lengths $3$ and $4$. For cycles of lengths $5$, $6$,
and $7$, it can be shown using the methods from the proof of Theorem~\ref{thm_mainthmpart3},
by considering a slightly modified definition of a shortcut, where vertices in some
of the pairs $\{i_j, i_j+1\}$ are merged. Thus, a minor modification of our argument
establishes the fact that a.a.s.~every subgraph $G'$ as above is $4$-pancyclic.
Moreover, we believe that it is true that a.a.s.~each such subgraph $G'$ is $3$-pancyclic.
In order to show this, it would be sufficient to prove a statement corresponding to
Lemma~\ref{lemma_saturation} for $3$-crossings (see Figure~\ref{fig:k-crossing}).
However, at this point we do not know how to prove such a statement.

\medskip

\noindent {\bf Acknowledgement. } This work was done during the IPAM Long Program, Combinatorics: Methods and Applications in Mathematics and Computer Science. We would like to greatly thank the members of staff at IPAM, the organizers of this program, and all its participants for creating a fantastic research environment. We would also like to thank our advisors J{\'o}zsef Balogh and Benny Sudakov, and the anonymous referee. 

\appendix

\section{Short Cycles without Hamiltonicity}

As mentioned in the concluding remark, when finding short cycles, i.e.,
cycles of length $3$ to $\delta(\varepsilon) n$, it is not necessary to
have the additional condition that the subgraph is Hamiltonian.
In the appendix, we prove the following theorem which establishes
this fact.

\begin{THM} \label{thm_mainthmpart1}
For all positive $\varepsilon$, there exists a constant $C$ such that if $p \geq Cn^{-1/2}$, then $G(n,p)$
a.a.s.~satisfies the following. Every subgraph
$G' \subset G(n,p)$ with more than $(\frac{1}{2}+\varepsilon)\frac{n^2p}{2}$ edges contains
a cycle of length $t$ for all $3 \leq t \leq \frac{\varepsilon}{25600}n$.
\end{THM}

The following well-known results from extremal graph theory will
be used in our proof.
First theorem allows
one to find a large minimum degree subgraph in a graph with many edges (see, e.g., \cite[Proposition 1.2.2]{Diestel}).
\begin{LEMMA} \label{lemma_findingmindegree}
Let $G$ be a graph on $n$ vertices with at least $dn$ edges. Then $G$ contains a subgraph $G'$ with minimum degree at least $d$.
\end{LEMMA}
Next lemma is P{\'o}sa's rotation-extension lemma (see, e.g., \cite[Ch. 10, Problem 20]{Lovasz}).
\begin{LEMMA} \label{lemma_posarotation}
Let $G$ be a graph such that $|N(X)\setminus X| \geq 2|X| -1$ for all $X \subset V(G)$
with $|X| \leq t$. Then for any vertex $v$ of $G$, there exists a path
of length $3t-2$ in $G$ that has $v$ as an end point.
\end{LEMMA}
\begin{REM}
The original statement does not say anything about the end point of the path.
However, this fact is implicit in the proof of this lemma given in \cite{Lovasz}.
\end{REM}

The proof of Theorem \ref{thm_mainthmpart1} follows the line of
Krivelevich, Lee, and Sudakov \cite{KrLeSu},
and most of the lemmas we use are variants of lemmas
from \cite{KrLeSu}. We will uses these lemmas to gain control
of subgraphs of a random graph. Since the modifications are quite simple, we will omit proofs of these lemmas.
First lemma, which is a variant of \cite[Lemma 3.3]{KrLeSu}, controls the growth of the neighborhood of a set.
\begin{LEMMA} \label{lemma_nbdgrowth}
For every $\varepsilon' \in (0,1)$, there exists a constant $C_0$ such that the following holds.
If $p = C n^{-1/2}$ for some constant $C \ge C_0$, and $r \in (0,1]$, then a.a.s.~every
subgraph $G' \subset G(n,p)$ satisfies the following property.
For all $X \subset V$ with $\abs{X} \ge \varepsilon' np$ and $\deg_{G'}(X) \geq \abs{X} rnp$,
we have $\abs{N_{G'}(X)} \geq (1 - \varepsilon')rn$.
\end{LEMMA}

Next lemma establishes the expansion property of subgraphs of random graphs
with large minimum degree. The same lemma appeared in \cite[Lemma 3.4]{KrLeSu}.
\begin{LEMMA} \label{lemma_largemindegreesubgraphexpands}
If $p = C n^{-1/2}$ for some constant $C> 0$, and $\varepsilon' > 0$, then a.a.s.
every subgraph $G' \subset G(n,p)$ with minimum degree at
least $\varepsilon' np$ satisfies the following expansion property.
For all $X \subset V$ with $|X| \leq \frac{1}{80}\varepsilon' n$, $|N_{G'}(X)\backslash X| \geq
2|X|$.
\end{LEMMA}

The final lemma is the key ingredient in the proof of Theorem \ref{thm_mainthmpart1}. It
asserts the existence of a vertex with many edges in
its 2nd neighborhood.
\begin{LEMMA} \label{lemma_findingspecialvertex}
For every $\varepsilon \in (0, \frac{1}{16})$, there exists a constant $C_0$ such that the following holds.
If $p = Cn^{-1/2}$ for some constant $C \ge C_0$, then $G(n,p)$
a.a.s.~satisfies the following. Every subgraph
$G' \subset G(n,p)$ with more than $(\frac{1}{2}+\varepsilon)\frac{n^2p}{2}$ edges contains a vertex $w$ such that
$e(N^{(2)}_{G'}(w)) \geq \frac{\varepsilon}{16} n^2p$.
\end{LEMMA}
\begin{proof}
Let $G=G(n,p)$ and $G'$ be a subgraph satisfying $e(G') \ge (\frac{1}{2}+\varepsilon)\frac{n^2p}{2}$.
By Chernoff's inequality (Theorem~\ref{thm_Chernoff}), we may assume that $\Delta(G') \le \Delta(G) \le (1 + \varepsilon)np$.
Let $B$ be the collection of all the vertices
which have degree at least $(\frac{1}{2} + \frac{\varepsilon}{2})np$ in $G'$.
Then, by the inequality
\begin{align*}
\left(\frac{1}{2} + \varepsilon\right)n^2p &\leq 2 e(G') = \deg_{G'}(V)
 = \deg_{G'}(B) + \deg_{G'}(V \setminus B)
\leq \deg_{G'}(B) + \left(\frac{1}{2} + \frac{\varepsilon}{2}\right)n^2p ,
\end{align*}
we have $\frac{\varepsilon}{2}n^2p \le \deg_{G'}(B) = \sum_{v \in V} \abs{ N_{G'}(v) \cap B }$. Therefore,
there should exist a vertex $v_0$ which satisfies $\abs{ N_{G'}(v_0) \cap B } \geq \frac{\varepsilon}{2} n p$.
Then $\deg_{G'}(N_{G'}(v_0) \cap B) \ge (\frac{1}{2} + \frac{\varepsilon}{2})np \abs{N_{G'}(v_0) \cap B}$, and
by Lemma \ref{lemma_nbdgrowth}, $\abs{ N_{G'}(N_{G'}(v_0) \cap B) } \geq (\frac{1}{2} + \frac{\varepsilon}{4})n$
for large enough $C_0$. Thus,
\[ \abs{ N^{(2)}_{G'}(v_0) } \geq \abs{N_{G'}(N_{G'}(v_0))} - \abs{N_{G'}(v_0)} - \abs{\{v_0\}} \geq \left(\frac{1}{2} + \frac{\varepsilon}{4}\right)n - (1 + \varepsilon)np - 1 \geq \left( \frac{1}{2} + \frac{\varepsilon}{8} \right)n \]
for large enough $n$. If $e(N^{(2)}_{G'}(v_0)) \geq \frac{\varepsilon}{16} n^2p$, then we have found the vertex $w = v_0$ as claimed, so assume otherwise.

\begin{figure}[h]
  \centering
  \begin{tabular}{ccc}

\begin{tikzpicture}
  \clip (-0.5,-1.5) rectangle (6.5,3);

  \draw [fill=black] (0,0) circle (0.5mm);
  \draw (0,-0.2) node {\footnotesize{$v_0$}};
  
  \draw [thick] (3.5,0) ellipse (2.3cm and 1cm);
  \draw (5.7,1) node {\footnotesize{$N^{(2)}(v_0)$}};

  \draw [thick] (3.2,0) ellipse (1.6cm and 0.7cm);
  \draw (5,0.3) node {\footnotesize{$X$}};
  
  \draw [thick] (1,2) ellipse (1cm and 0.5cm);
  \draw (0.1,2.5) node {\footnotesize{$B$}};
  
  \draw [thick] (0.8,2) ellipse (0.5cm and 0.25cm);
  \draw [<-] (1.3,2.2) -- (2.6,2.5);
  \draw (2.6,2.5) node [right] {\footnotesize{$N(v_0) \cap B$}};
  
  \draw (0,0) -- (1.2746,1.9213);
  \draw (0,0) -- (0.3014,2.0186);
  
  \draw (1.1,2.2) -- (4.99,0.7618);
  \draw (0.3114,1.9468) -- (1.2401,-0.186);
\end{tikzpicture} & \hspace{0.1in} & 

\begin{tikzpicture}
  \clip (-0.5,-1.5) rectangle (7.5,3);

  \draw [fill=black] (5,1.8) circle (0.5mm);
  \draw (5.2,1.6) node {\footnotesize{$v_1$}};

  \draw [thick] (1,0.7) ellipse (1.2cm and 2cm);
  \draw (0.1,2.4) node {\footnotesize{$X$}};

  \draw [thick] (5,0.7) ellipse (1.2cm and 2cm);
  \draw (4.1,2.4) node {\footnotesize{$Y$}};

  \draw [thick] (5,0.05) ellipse (0.6cm and 1.2cm);
  \draw [<-] (5.5,-0.8) -- (6.1,-1);
  \draw (6.1,-1) node [right] {\footnotesize{$N^{(2)}(v_1)$}};

  \draw [thick] (1,0.05) ellipse (0.5cm and 0.9cm);
  \draw (0.4,0.9) node {\footnotesize{$B_X$}};

  \draw [thick] (1,0.3) ellipse (0.25cm and 0.5cm);
  \draw [<-] (0.95,0.85) -- (0.8,1.5);
  \draw (0.8,1.5) node [above,xshift=2mm] {\footnotesize{$N(v_1) \cap B_X$}};

  \draw (5,1.8) -- (1.0611,-0.1848);
  \draw (5,1.8) -- (0.9691, 0.7962);

  \draw (0.9704,-0.1965) -- (4.9289,-1.1415);
  \draw (0.9859,0.7992) -- (4.9662,1.2481);
\end{tikzpicture}
  \end{tabular}
  \caption{The second neighborhood of either $v_0$ or $v_1$ must contain many edges.}
\end{figure}
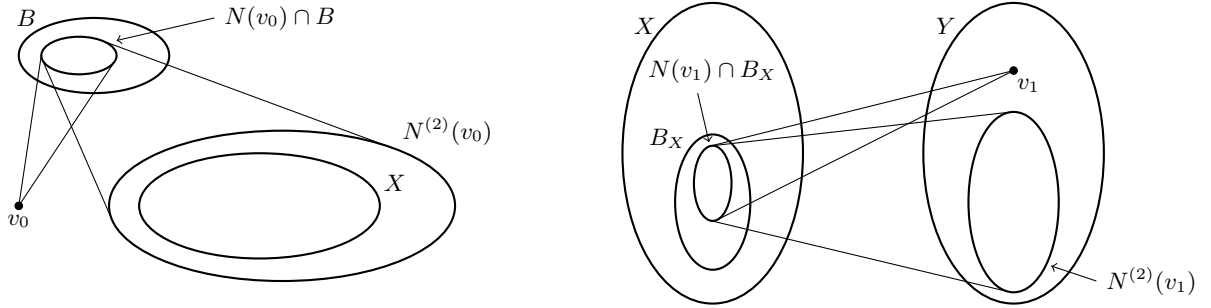

Let $X \subset N^{(2)}_{G'}(v_0)$ be such that $|X| = (\frac{1}{2} + \frac{\varepsilon}{80})n$ and $Y = V(G) \setminus X$
so that $\abs{ Y } = (\frac{1}{2} - \frac{\varepsilon}{80})n$.
Let $G''$ be the graph obtained from $G'$ by removing all the edges within $X$.
Then by assumption, we have
\[ e(G'') \geq e(G') - \frac{\varepsilon}{16} n^2p \geq \left(\frac{1}{2}+\frac{\varepsilon}{2} \right)\frac{n^2p}{2}. \]
By Chernoff's inequality (Theorem~\ref{thm_Chernoff}), a.a.s.~$e_{G}(Y) \le \frac{n^2p}{8}$ and $e_G(X,Y) \leq (1 + \frac{\varepsilon}{64}) \abs{ X } \abs{ Y } p$.
Therefore, we can find $r \in [0, \frac{1}{4} + \frac{\varepsilon}{64}]$ satisfying $e_{G''}(X,Y) = rn^2p$. Moreover,
\[ \frac{n^2p}{4} \le e(G'') = e_{G''}(X) + e_{G''}(X,Y) + e_{G''}(Y) \leq \frac{\varepsilon}{16} n^2p + rn^2p + \frac{n^2p}{8}, \]
and thus $r > 1/16 \ge \varepsilon$. Let $B_X \subset X$ be the collection of vertices in $X$ which have
at least $(2r - \frac{\varepsilon}{20})np$ neighbors in $Y$. Then the following inequality holds:
\begin{align*}
rn^2p &= e_{G''}(X,Y) = \sum_{v \in X} \abs{ N_{G''}(v) \cap Y } = \sum_{v \in X} \abs{ N_{G''}(v) } = \deg_{G''}(B_X) + \deg_{G''}(X \setminus B_X) \\
&\leq \deg_{G''}(B_X) + \left(2r - \frac{\varepsilon}{20}\right)np |X| = \deg_{G''}(B_X) + \left(2r - \frac{\varepsilon}{20}\right)\left(\frac{1}{2} + \frac{\varepsilon}{80}\right)n^2p.
\end{align*}
Subtracting $\left(r - \frac{\varepsilon}{40}\right)n^2p$ from each side gives $\frac{\varepsilon}{40}n^2p \leq \deg_{G''}(B_X) +  \frac{\varepsilon}{80}\left(2r - \frac{\varepsilon}{20}\right)n^2p$,
which implies that $\deg_{G''}(B_X) \geq \frac{\varepsilon}{80}n^2p$ because of the restriction $r \le \frac{3}{8}$ and $\varepsilon < \frac{1}{16}$.
Moreover, since
\[ \sum_{y \in Y} \abs{ N_{G''}(y) \cap B_X } = \deg_{G''}(B_X) \geq \frac{\varepsilon}{80}n^2p, \]
there should exist a vertex $v_1 \in Y$ which satisfies $\abs{ N_{G''}(v_1) \cap B_X } \geq \frac{\varepsilon}{80}n p$.
Then by Lemma \ref{lemma_nbdgrowth},
$\abs{ N_{G''}(N_{G''}(v_1) \cap B_X) } \geq (1 - \frac{\varepsilon}{20})(2r - \frac{\varepsilon}{20})n \geq (2r - \frac{\varepsilon}{10})n$
for large enough $C_0$.
Moreover, $N_{G''}(N_{G''}(v_1) \cap B_X) \subset N^{(2)}_{G''}(v_1) \cap Y$ since we removed all the edges within $X$. Thus,
\begin{align*}
 \abs{ N^{(2)}_{G''}(v_1) \cap Y } &\geq \abs{N_{G''}(N_{G''}(v_1) \cap B_X)} - \abs{N_{G''}(v_1)} - \abs{\{v_1\}} \\
&\geq \left(2r - \frac{\varepsilon}{10}\right)n - (1+\varepsilon)np - 1 \geq \left(2r - \frac{\varepsilon}{8} \right)n
\end{align*}
for large enough $n$. If $e(N^{(2)}_{G''}(v_1)) \geq \frac{\varepsilon}{16} n^2p$ then we have found the vertex $w = v_1$ as required.
We claim that this should always be the case.
Otherwise, let $Y' \subset N^{(2)}_{G''}(v_1) \cap Y$ be a set of size $\abs{Y'} = (2r - \frac{\varepsilon}{8})n$ and
note that $\abs{Y \setminus Y'} \geq (\frac{1}{2} - \frac{\varepsilon}{80})n - (2r - \frac{\varepsilon}{8})n \geq \frac{\varepsilon}{32}n$.
Since both $Y \setminus Y'$ and $Y'$ are sets of linear size, we can use Chernoff's inequality (Theorem~\ref{thm_Chernoff}) to bound $e_{G''}(Y \setminus Y')$ and $e_{G''}(Y \setminus Y', Y')$
to get:
\begin{align*}
\left(\frac{1}{2}+\frac{\varepsilon}{2}\right)\frac{n^2p}{2} &\le e_{G''}(V) = e_{G''}(X) + e_{G''}(X,Y) + e_{G''}(Y) \le 0 + rn^2p + e_{G''}(Y) \\
&= rn^2p + e_{G''}(Y \setminus Y') + e_{G''}(Y \setminus Y', Y') + e_{G''}(Y') \\
&\le rn^2p + \left(1 + \frac{\varepsilon}{16} \right) \abs{Y \setminus Y'}^2\frac{p}{2} + \left(1 + \frac{\varepsilon}{16} \right) \abs{Y \setminus Y'}\abs{Y'}p + \frac{\varepsilon}{16} n^2p \\
&\le \frac{\varepsilon}{16} n^2p + rn^2p + \left(1 + \frac{\varepsilon}{16} \right) \left(\abs{Y}^2 - \abs{Y'}^2 \right)\frac{p}{2} \\
&\le  \frac{\varepsilon}{16} n^2p + rn^2p + \left(1 + \frac{\varepsilon}{16} \right) \left(\frac{1}{4} - \left(2r - \frac{\varepsilon}{8}\right)^2 \right)\frac{n^2p}{2} .
\end{align*}
Divide each side by $\frac{n^2p}{2}$ to get
\begin{align*}
\frac{1}{2}+\frac{\varepsilon}{2}
 &\le \frac{\varepsilon}{8} + 2r + \left(1 + \frac{\varepsilon}{16} \right) \left(\frac{1}{4} - \left(2r - \frac{\varepsilon}{8}\right)^2 \right) \le \frac{\varepsilon}{8} + 2r + \left(\frac{1}{4} - \left(2r - \frac{\varepsilon}{8}\right)^2 \right) + \frac{\varepsilon}{16}, \\
&\le \frac{3\varepsilon}{16} + 2r + \left(\frac{1}{4} - 4r^2 + \frac{r\varepsilon}{2}- \frac{\varepsilon^2}{64} \right) \le \frac{3\varepsilon}{16} + 2r + \frac{1}{4} - 4r^2 + \frac{3\varepsilon}{16}- \frac{\varepsilon^2}{64},
\end{align*}
which when rearranged gives,
$ \left( \frac{1}{2} - 2r \right)^2 + \frac{\varepsilon}{8} + \frac{\varepsilon^2}{64} \le 0$; this is a contradiction, since $\varepsilon > 0$. Therefore, either the second neighborhood of $v_0$ or $v_1$ should have had at least $\frac{\varepsilon}{16} n^2p$ edges inside it.
\end{proof}

Now we are ready to show the existence of short cycles.
\begin{proof}[Proof of Theorem \ref{thm_mainthmpart1}]
By Proposition \ref{prop_probrestriction} we may assume that $p = Cn^{-1/2}$ for some constant $C$ to
be chosen later, and prove only this case.
Moreover, we may assume that $\varepsilon \le \frac{1}{16}$, since for larger $\varepsilon$, the statement follows from the case $\varepsilon = \frac{1}{16}$.
Let $G = G(n,p)$ and $G'$ be a subgraph satisfying $e(G') \geq (\frac{1}{2} + \varepsilon)\frac{n^2p}{2}$.
By Chernoff's inequality (Theorem~\ref{thm_Chernoff}), we know that a.a.s.~every pair of vertices has
codegree less than $(\log n) np^2 = C^2\log n$ in $G$, thus
throughout the proof we will assume that this estimate holds.
We can find cycles of length 3 by using Haxell, Kohayakawa, and {\L}uczak's Theorem \ref{thm_resilientsmallcycle}
and cycles of length 4 by Bondy and Simonovits' Theorem \ref{thm_bondysimonovits} by choosing $C = C(\varepsilon)$ to be large enough.

To find a cycle of length $t$ for some $5 \leq t \leq \frac{\varepsilon}{25600}n$, first apply Lemma \ref{lemma_findingspecialvertex} to
find a vertex $w$ which satisfies $e(N^{(2)}_{G'}(w)) \ge \frac{\varepsilon}{16} n^2p$. Then by Lemma \ref{lemma_findingmindegree},
there exists a subset $Z \subset N^{(2)}_{G'}(w)$ such that $G'[Z]$ has minimum degree at least $\frac{\varepsilon}{16}np$.
Pick an arbitrary vertex $w_2 \in Z$, and pick a vertex $w_1 \in N_{G'}(w)$ so that $w w_1 w_2$ is a path of length 2
in $G'$. Then remove all the neighborhoods of $w_1$ in $Z$ except $w_2$ and let $Z_1$
be the resulting set. Since the codegree of every pair of vertices
is less than $C^2\log n$, $G'[Z_1]$ will be a graph of minimum degree at least
$\frac{\varepsilon}{16} np - C^2 \log n$, which is at least $\frac{\varepsilon}{32}np$ for large enough $n$. Thus by Lemma \ref{lemma_largemindegreesubgraphexpands},
the graph $G'[Z_1]$ has an expansion property; namely, every set $X \subset V$ of size
 at most $|X| \leq \frac{\varepsilon}{25600}n$ satisfies $|N_{G'}(X)\backslash X| \geq
2|X|$. By Lemma \ref{lemma_posarotation}, we can find
a path of length $\frac{\varepsilon}{25600}n$ in $G'$ which has $w_2$ as an endpoint. Call this path
$w_2x_1x_2\ldots x_{\varepsilon n/25600}$. Note that by the definition of $Z_1$, 
for each $x_s$ in this path, there exists a vertex
$x_s'$ in $N_{G'}(w)$ that is not $w_1$ which forms a path $wx_s'x_s$ in $G'$.
Then $ww_1w_2x_1 \ldots x_{s} x_{s}'w$ forms a cycle of length $s+4$ in $G'$.
Thus, we have found cycles of length $t$ for all $5 \leq t \leq \frac{\varepsilon}{25600}n$.
\end{proof}

\bibliographystyle{plain}

\end{document}